\documentclass{article}
\usepackage{amsmath}
\usepackage{amssymb}
\usepackage{amsthm}
\usepackage[all,knot]{xy}

\newcommand{\cA}{\mathcal{A}}
\newcommand{\cL}{\mathcal{L}}

\newcommand{\Q}{\mathbb{Q}}
\newcommand{\N}{\mathbb{N}}

\newcommand{\ta}{t^{12}}
\newcommand{\tb}{t^{23}}
\newcommand{\tc}{t^{13}}
\newcommand{\tf}{t^{14}}
\newcommand{\td}{t^{24}}
\newcommand{\te}{t^{34}}

\newcommand{\lx}{\partial_x\lambda(x,y)}
\newcommand{\ly}{\partial_y\lambda(x,y)}
\newcommand{\la}{\partial_a\lambda(a,b)}
\newcommand{\lb}{\partial_b\lambda(a,b)}
\newcommand{\lc}{\partial_c\lambda(c,b)}
\newcommand{\FF}{\mathit{1} \mathit{F} \mathit{F}}
\newcommand{\TFF}{\mathit{2} \mathit{F} \mathit{F}}
\newcommand{\FT}{\mathit{4} T}
\newcommand{\tij}{t^{ij}}
\newcommand{\tjk}{t^{jk}}
\newcommand{\tki}{t^{ki}}
\newcommand{\tkl}{t^{kl}}
\newcommand{\ti}{t^{4i}}
\newcommand{\tj}{t^{4j}}
\newcommand{\tk}{t^{4k}}

\newcommand{\hl}{[\hat{l}]}
\newcommand{\AFF}{\cA^{\FF}}
\newcommand{\ATFF}{\cA^{\TFF}}
\newcommand{\LCC}{\cL^{CC}}
\newcommand{\Lfour}{[[\cL,\cL],[\cL,\cL]]}

\newtheorem{remark}{Remark}
\newtheorem{theorem}{Theorem}
\newtheorem{proposition}{Proposition}

\newtheorem{lemma}{Lemma}

\begin{document}
\title{Closed-Form Associators in Permutative Chord Diagrams}
\author{Peter Lee}
\maketitle

\begin{abstract}
Construction of a universal finite-type invariant can be reduced,
under suitable assumptions, to the solution of certain equations
(the hexagon and pentagon equations) in a particular graded associative
algebra of chord diagrams.  An explicit, closed-form solution to
these equations may, indirectly, give information about various
interesting properties of knots, such as which knots are ribbon.
However, while closed-form solutions (as opposed to solutions which
can only be approximated to successively higher degrees) are needed
for this purpose, such solutions have proven elusive, partly as a
result of the non-commutative nature of the algebra.  To make the
problem more tractable, we restrict our attention to solutions of
the equations in the subalgebra of horizontal chord diagrams, viewed as a graded unital permutative algebra -- where `permutative' means that $u[x,y]=0$ whenever $u$ has degree $\geq 1$. We show that this
restriction leads in a straightforward and fairly short way to a
reduction of the hexagon and pentagon equations to a simpler
equation taken over the algebra of power series in two commuting
variables. This equation had been found and solved explicitly by
Kurlin under a superficially different set of reduction assumptions,
which we show here are in fact equivalent to ours.  This paper thus
provides an alternative (simpler and shorter) derivation of Kurlin's equation.
\end{abstract}

\tableofcontents

\section{Introduction}

The fundamental theorem of finite-type invariants assures us that a
universal finite type invariant exists which takes values in the
associative algebra of chord diagrams on the circle (modulo the
4-term relation).  Construction of such an invariant can be reduced,
under suitable assumptions, to solving certain known equations (the
hexagon and pentagon equations) in the related algebra $\cA_n$ of
chord diagrams on $n$ vertical strands (modulo similar relations).
Solving these equations explicitly has proved to be extremely
difficult, partly because multiplication in $\cA_n$ is associative
but not commutative.

A universal formula for solutions to the equations has been given by
Drinfel'd \cite{Drinfel'd2} using integral methods.  However, for
applications, closed-form solutions (as opposed to solutions which
can only be approximated term-by-term) are often needed.  For
instance, it is expected that finite-type invariants can be
constructed which may give information about the genus of knots, or
which knots are ribbon knots \cite{DBN6}.  But finite degree
approximations are not generally expected to suffice in this regard
--- in particular, Ng \cite{Ng} has shown that no finite degree approximation to a
universal finite-type invariant can give information about ribbon
knots.

Lieberum \cite{Lieberum} has considered the pentagon and hexagon equations in a related context in which solutions are required to belong to (tensor powers of) the universal enveloping
algebra of gl(1$|$1), and found a closed--form solution. However, the difficulty of solving the
equations explicitly for universal finite-type invariants has
prompted researchers to consider various simplified versions of the
problem. One approach, proposed by Drinfel'd in \cite{Drinfel'd2},
is to assume solutions to the equations are group-like elements of
the relevant associative algebra, and to consider the image of the
equations under the logarithm function. The result is certain
equations over a Lie algebra $\cL$ which has the same generators as
the associative algebra. Drinfel'd discussed the solution of these
equations in the quotient $\cL / [[\cL,\cL],[\cL,\cL]]$, without
however giving explicit formulae.

Kurlin \cite{Kurlin} subsequently took up Drinfel'd's approach and
achieved a significant simplification. Specifically he showed that
the logarithmic hexagon and pentagon equations, in the given
quotient of the Lie algebra $\cL$ (which he called `compressed'
equations), were equivalent to a single equation, this time taken
over the algebra of ordinary power series in two formal commuting
variables $x$ and $y$:

\begin{equation}
\lambda(x,y)+e^{-y}\ \lambda(y,z)+e^x\ \lambda(x,z)\ =\
\frac{1}{xy}+\frac{e^{-y}}{yz}+\frac{e^x}{xz} \label{SimplifiedEq}
\end{equation}
Here $\lambda$ is a power series in $x$ and $y$ (for which we are
solving) and $z$ stands for $-(x+y)$.

Translating equations over a Lie algebra into an equation over an
ordinary commuting power series algebra constituted an important
development. Furthermore Kurlin was also able to solve this equation and
specify explicitly all of the solutions.

In this paper we take an alternative approach to simplifying the
hexagon and pentagon equations, which turns out to be equivalent but
is simpler and easier to extend. Like Kurlin we restrict ourselves
to the algebra of chord diagrams in which all chords are horizontal.
In the relevant case where our diagrams lie on three vertical
strands, we essentially get a power series algebra in three
non-commuting variables subject to certain additional relations.
Unlike Kurlin, we consider those equations at the associative
algebra level, but over the quotient algebra in which:

\begin{equation*}
u[x,y] = 0
\end{equation*}
where $x$ and $y$ are any elements of the algebra, and $u$
represents a product of one or more chords (we call this the `one
frozen foot', or 1FF, quotient). In a later part of the paper,  we
generalize somewhat to the case where $u$ is a product of any two or
more chords (the two frozen feet, or 2FF, quotient).

We note that passing to the 1FF quotient amounts to viewing the positive degree part of the horizontal chord diagram algebra as an algebra over the operad `Perm' -- that is, an associative algebra in which which $xyz=xzy$ for all elements $x,y$ and $z$ of the algebra (see \cite{Zin}).

We show that the hexagon and pentagon equations, modulo 1FF, again
reduce to Kurlin's equation defined over the algebra of commuting
power series in two variables. Moreover, we show that solving
Kurlin's compressed equations at the Lie algebra level is equivalent
to solving the hexagons and pentagon at the associative algebra
level modulo 1FF. Thus in particular we get an alternative, but
simpler and shorter, derivation of Kurlin's simplified equation
(\ref{SimplifiedEq}).

We also show how to solve the hexagon and pentagon equations in the
2FF quotient and show that all such solutions are derived, in an
explicitly given way, from solutions over 1FF.

Part 2 of the paper is organized as follows: in Section 2.1, we
review various relevant algebras of chord diagrams and their related
Lie algebras. In Section 2.2, we set forth the equations we wish to
solve and the assumptions we place $ab\ initio$ on the solutions. In
Section 2.3 we derive some basic properties of the chord diagram
algebras and of solutions to the hexagon and pentagon equations,
based on the given assumptions.  In Section 2.4 we show how the
hexagon and pentagon equations, in the 1FF quotient, reduce to
Kurlin's equation (\ref{SimplifiedEq}).  We also show that the
hexagon and pentagon equations hold (with group-like solutions)
modulo 1FF if and only if their logarithmic images hold (with Lie
series solutions) modulo $[[\cL,\cL],[\cL,\cL]]$, thus establishing
the equivalence of our results with the results of Kurlin. Finally,
in Section 2.5 we extend the main result of Section 2.4 to two
frozen feet. Specifically, we show how the hexagon and pentagon
equations can be solved in the 2FF quotient by a specific extension
of solutions in the 1FF quotient.

\paragraph{Acknowledgement.}  The author would like to thank Dror
Bar-Natan for suggesting the project which forms the subject of this
paper, and for many helpful discussions.

\section{Closed-Form Associator}

\subsection{Algebras of Chord Diagrams}

\subsubsection{The Algebras $\hat{\cA}$ and $\AFF$}

In this paper, we consider the algebra
\begin{equation}
\hat{\cA} :=\ \widehat{\cA ^{hor}}
\end{equation}
where $\cA^{hor}$ refers to the algebra of chord diagrams on $n$
vertical strands (modulo the 4T relation - see below), but in which
we allow only horizontal chords.  The $n$ is usually implicit, but
sometimes it is indicated, as in $\cA_n$.  The hat on the RHS above
means we take the formal completion. We can also define

\begin{equation}
\AFF := \hat{\cA} /\ \FF,
\end{equation}
where, again, 1FF stands for `1 Frozen Foot', i.e. allowing all
chords after the first (counting from the bottom) to commute.  In a
later portion of the paper we consider the quotient $\ATFF =
\hat{\cA} /\ \TFF$, in which all chords after the second (counting
from the bottom) commute.

All algebras in this paper are considered over the ground field
$\Q$.

\subsubsection{Notation}

In this paper we will mostly be concerned with the case $n = 3$. The
resulting algebra can be expressed in purely algebraic terms as an
(almost) commutative polynomial algebra on three letters $a, b, c$:
\begin{equation}
\hat{\cA} \ :=\ \mathbb{Q}\ll a,b,c \gg\ /\ \FT, \FF
\end{equation}
where for convenience we use the convention (followed through much
of the paper):

\[
\xy (-30,0)*{\xy (-5,0)*{a}; (0,0)*{=}; (8,0)*{\xy (-2,3)*{};
(-2,-3)*{}
**\dir{-}; (2,3)*{}; (2,-3)*{} **\dir{-}; (6,3)*{}; (6,-3)*{} **\dir{-}; (-2,0)*{}; (2,0)*{}
**\dir{.} \endxy} \endxy};
(0,0)*{\xy (-5,0)*{b}; (0,0)*{=}; (8,0)*{\xy (-2,3)*{}; (-2,-3)*{}
**\dir{-}; (2,3)*{}; (2,-3)*{} **\dir{-}; (6,3)*{}; (6,-3)*{} **\dir{-}; (2,0)*{}; (6,0)*{}
**\dir{.} \endxy}
\endxy};
(30,0)*{\xy (-5,0)*{c}; (0,0)*{=}; (8,0)*{\xy (-2,3)*{}; (-2,-3)*{}
**\dir{-}; (2,3)*{}; (2,-3)*{} **\dir{-}; (6,3)*{}; (6,-3)*{} **\dir{-}; (-2,0)*{}; (6,0)*{}
**\dir{.} \endxy}
\endxy}
\endxy
\]

The 4-term relation 4T here translates to

\begin{equation*}
[a \negthinspace +b \negthinspace, c]\ =\ [b\negthinspace+c,a]\ =\
[c \negthinspace +a,b] =\ 0
\end{equation*}
or equivalently

\begin{equation*}
[a \negthinspace + b \negthinspace +c, w] =0, \ \forall w \in \cA
\end{equation*}
and 1FF translates to:

\begin{equation*}
uxy\ =\ uyx\\
\end{equation*}
where

\begin{equation*}
u,x,y \in \mathcal{A}\ \text{and}\ deg\ u \geq 1
\end{equation*}

\subsubsection{Operations on $\hat{\cA}$}

There are $n$ `degeneracy' operations $\eta_i, \ i \in \{1, \dotsc,
n\}$, defined on $\cA_n$ (and therefore on $\hat{\cA_n}$) by the
fact that $\eta_i$ acts on individual chord diagrams by deleting the
$i$-th strand (and relabeling subsequent strands) if no chord ends
on the $i$-th strand, and sending the diagram to zero if any chord
ends on the $i$-th strand.

There are also strand doubling operations $\Delta_i$, where $i$ may
range over the number of strands in a diagram.  $\Delta_i$ acts by
doubling the $i$-th strand, and summing over the number of ways of
`lifting' all strands previously ending on the $i$-th strand to the
two new strands.  For instance, viewing $a$ as an element of $\cA_2$
in the obvious way, $\Delta_1(a) = b+c$.  In Part 3 of the paper, we
will only allow doubling of the first (shield) strand, that is we
only allow the doubling operation $\Delta_1$.

When $\Delta_i$ acts on a diagram with just one strand, we write
$\Delta$.  This gives rise to the notation $\Delta_i = (1\otimes
\dotsb \otimes \Delta \otimes \dotsb \otimes 1)$ where the $\Delta$
is in the $i$-th position.  We sometimes write this $(1 \dotso 1
\Delta 1 \dotso 1)$, omitting the tensor symbols.

The symmetric group $\mathbb{S}_n$ acts on diagrams with $n$ strands
by permuting the strands.  If $D$ is a diagram with $n$ strands, we
indicate the action of the permutation which sends $(12 ... n)$ to
$(xy ... z)$ on $D$ as $D^{xy \dotso z}$.  For instance, we have

\begin{equation*}
a^{132} = c
\end{equation*}

\subsubsection{Induced Lie Algebras}

The algebras $\cA_n$ (and their completions) induce Lie algebras
$\mathcal{L}_n$ (and their completions $\hat{\mathcal{L}_n}$) with
generators the single-chord diagrams in $\cA_n$, with product the
commutator bracket: $[x,y] := xy - yx$, and with the Lie ideal of relations generated by the defining relations for the $\cA_n$.

We can also form the quotient

\begin{equation*}
\LCC:= \hat{\cL} /\ [[\cL,\cL],[\cL,\cL]]
\end{equation*}
where the superscript `$CC$' stands for `commutators commute'.

\subsection{Criteria for $\Phi$}

Following Drinfel'd \cite{Drinfel'd2}, we are looking for invertible $R \in \hat{\cA}_2$ and $\Phi \in
\hat{\cA}_3$ satisfying the following primary conditions:

\begin{itemize}

\item The Hexagons

Positive Hexagon:
\begin{equation}
\label{hex}
(\Delta 1)R\ =\ \Phi \cdot R^{23} \cdot (\Phi^{-1})^{132} \cdot R^{13} \cdot \Phi^{312}\\
\end{equation}

Negative Hexagon:
\begin{equation}
\label{hexmin}
(\Delta 1)(R^{-1})\ =\ \Phi \cdot (R^{-1})^{23} \cdot (\Phi^{-1})^{132} \cdot (R^{-1})^{13} \cdot \Phi^{312}\\
\end{equation}

\item The Pentagon
\begin{equation}
\label{pentagon} \Phi^{123} \cdot (1\Delta 1)\Phi \cdot \Phi^{234} =
(\Delta 11) \Phi \cdot (11 \Delta) \Phi
\end{equation}

\end{itemize}

We also impose the following ancillary conditions:

\begin{itemize}

\item Symmetry

\begin{equation}
\Phi \cdot \Phi^{321} = 1 \label{unitary}
\end{equation}

\item Non-Degeneracy

\begin{equation}
\eta_1 \Phi = \eta_2 \Phi = \eta_3 \Phi = 0 \label{nondegeneracy}
\end{equation}

\item Group-like

\begin{equation}
\Phi = \exp(\phi), \text{ for some } \phi \in \hat{\cL}_3
\label{grouplike}
\end{equation}

\end{itemize}

It can be shown that (under the assumptions in this paper, and
particularly the restriction to horizontal chords, which forces $R$
to be symmetric) any $\Phi$ which satisfies the hexagons must
satisfy the symmetry (or `unitarity') condition (see
\cite{Drinfel'd2} at equation (2.10) $et\ seq.$, or \cite{DBN2}
Prop. 3.7).

We also take $R$ to have the form

\begin{equation}
\label{rmatrix} R=\exp(a)
\end{equation}
where $a$ is viewed as an element of $\cA_2$.

\subsection{Some Basic Manipulations in $\AFF$}

We are looking for a non-commutative power series $\Phi(a,b)$ which
is group-like in the sense that $\Phi(a,b)=\exp(\phi(a,b))$, where
$\phi(a,b)$ is a Lie series in $a$, $b$, i.e. $\phi(a,b) \in
\hat{\cL_3}$. In fact, we take the terms of $\phi$ to be of degree
at least two, so that the leading term of $\phi(a,b)$ is $[a,b]$ (up
to some multiplicative factor). We will now derive a number of
simple consequences of the assumptions we have made for $\AFF$,
$\phi$ and $\Phi$.

\begin{remark}
Lie Word Notation \end{remark} The following notation will be
useful: for $w$ a word in the alphabet $\{a,b\}$, that is $w=w_1 w_2
\ldots w_n$ for $w_i \in \{a,b\}$, we write $[w]$ for

\begin{equation*}
[w_1,[w_2,[ \ldots [w_{n-1},w_n] \ldots ]]]
\end{equation*}
However we do sometimes still write the commutator of two terms as
$[x,y]$.

\begin{lemma}
Using the definition of $a,b,c$ above,

(i) $[ab]=[bc] = [ca]$

(ii) Modulo 1FF, we have

\begin{equation*}
[ab]c = [ab](-a-b)
\end{equation*}
and more generally

\begin{equation*}
[ab]c^n = [ab] (-a-b)^n
\end{equation*}
Equivalently, in any expression which is pre-multiplied by the
commutator $[ab]$, we can take $c=(-a-b)$ modulo 1FF.

(iii) $[c^n ab] = [(-a-b)^n ab]$ in $\hat{\cA}$.
\label{eliminatec}
\end{lemma}

\begin{proof}
(i) This follows from the centrality of $a+b+c$, i.e.
$[a,c]=[a,c-a-b-c] = [a,-b] = [b,a]$ and similarly for the other
equality.

(ii) This is just a simple calculation.

(iii) This is another immediate consequence of the centrality of
$a+b+c$.
\end{proof}

\begin{proposition}
Modulo $[[\cL,\cL],[\cL,\cL]]$ (and hence also modulo 1FF and 2FF,
as these subalgebras contain $[[\cL,\cL],[\cL,\cL]]$),

(i) for $w$ a word on $\{a,b\}$ of length at least two (so that
$[w]$ is a commutator), we have $[abw]=[baw]$;

(ii) for $w$ a word on $\{a,b\}$ with $n \ a$'s and $m\ b$'s, we
have $[wab] = [a^nb^mab]$;

(iii) any Lie series $l(a,b,c)$ beginning in degree two (or higher)
can be written
\begin{equation}
l(a,b,c) =  [w(a,b)ab]
\end{equation}
where $w(a,b)$ is a unique commutative power series of words on
$\{a,b\}$ and the notation $[wab]$ has been extended in the obvious
way; and

(iv) for Lie series $l(a,b)$ and $l'(a,b)$, the following are
equivalent:

\begin{align*}
l = l' mod\ 1FF \\
l = l' mod\ [[\cL,\cL],[\cL,\cL]]
\end{align*}
\label{lieseries}
\end{proposition}

\begin{proof}
The Jacobi relation gives us
\begin{equation}
[abw]=-[b[w]a] -[[w]ab] = -[b[w]a] = [ba[w]] \notag
\end{equation}

The remaining assertions follow immediately.
\end{proof}

The following is clear from the previous proposition.

\begin{proposition}
Modulo 1FF, if $w$ is a word on $\{a,b\}$ with $n \ a$'s and $m\
b$'s, we have $[wab] = (-1)^{n+m}[ab]a^nb^m.$ \label{liemonomials}
\end{proposition}

\begin{remark}
Lie Algebra Decomposition and Vector Space Bases

It follows from Lemma \ref{eliminatec} and Propositions
\ref{lieseries} and \ref{liemonomials} that:

(i) As Lie algebras, $\cL \cong \Q c \oplus \cL (a,b)$, where $\cL
(a,b)$ is the Lie algebra generated over $\Q$ by $a,b$ modulo 4T.

(ii) $\LCC = \hat{\cL}/ [[\cL,\cL],[\cL,\cL]]$ has vector space
basis $\{a,b,c\}\ \cup\ \{[a^n b^m ab]: \ n,m \geq 0\}$.
\end{remark}

This remark recaptures some results from \cite{Kurlin}, see in
particular Proposition 3.4 thereof.

We now explore the implication of the symmetry requirement
(\ref{unitary}) for the form $\Phi$ must take.

\begin{proposition}
Taking $\Phi(a,b) = \exp \phi(a,b)$ where $\phi(a,b) = [w(a,b)ab]$
and $w(a,b)$ is a power series in $a$ and $b$, the symmetry
requirement (\ref{unitary}) is equivalent to $w(a,b)$ being
symmetric. \label{unitaryPhi}
\end{proposition}

\begin{proof}
A key simplification valid in both 1FF and 2FF is the following
`linearization' effect:

\begin{align*}
\exp([w(a,b)ab]) & = 1 + [w(a,b)ab] + 1/2! [w(a,b)ab]^2 + \dotso \\
& = 1 + [w(a,b)ab]
\end{align*}

As a result we will look for $\Phi$ of the form:

\begin{equation}
\Phi(a,b) = 1 + [w(a,b)ab] \label{phi}
\end{equation}

Now $\Phi$ must also satisfy the `symmetry' (or `unitarity')
property, $\Phi^{-1} = \Phi^{321}$. Again using the `linearization'
effect valid in 1FF or 2FF we have:

\begin{align}
\Phi(a,b)^{-1} & = 1 - [w(a,b)ab] + [w(a,b)ab]^2 - \dotso \notag \\
&= 1 - [w(a,b)ab] \label{Phiinverse}
\end{align}

Note that if the strands in a chord diagram are permuted according
to the `unitary' transformation $(123)\ \rightarrow\ (321)$ the
chords $a,b,c$ transform as follows:

\begin{align*}
a\ \rightarrow\ b\\
b\ \rightarrow\ a\\
c\ \rightarrow\ c
\end{align*}

So we get:

\begin{align*}
\Phi(a,b)^{321} = & 1 + [w(b,a)ba] \\
= & 1 - [w(b,a)ab]
\end{align*}

Comparing the expressions for $\Phi^{-1}$ and $\Phi^{321}$, we see
that equation (\ref{unitary}) is equivalent to $w$'s being
symmetric.
\end{proof}

The following lemma further refines the form that a Lie series may
take over 1FF or 2FF.

\begin{lemma}
Modulo 2FF, a Lie series $[w(a,b)ab]$ where $w(a,b)$ is an ordinary
power series in commutative variables ${a,b}$ can be written:

\begin{equation}
[w(a,b)ab] = [ab]\lambda(a,b) - a[ab]\partial _a \lambda(a,b) -
b[ab]\partial _b \lambda(a,b) \label{formofw2FF}
\end{equation}
where $\lambda(a,b) := w(-a,-b)$.
Modulo 1FF this reduces to:

\begin{equation}
[w(a,b)ab] = [ab]\lambda(a,b) \label{formofw1FF}
\end{equation}
\label{lemmaformofw}
\end{lemma}

\begin{proof}
It is enough to check the case where $\lambda$ is a monomial.  We
have
\begin{equation*}
(-1)^n[a^nab] = [ab]a^n - a[ab]na^{n-1}
\end{equation*}

Indeed, the result is clear for $n=1$, and for $n>1$ we have:

\begin{align*}
(-1)^n[a^nab] &= (-1)^n[a \cdot a^{n-1}ab] \\
&= (-1)\bigl\{a[ab]a^{n-1} - a^2[ab](n-1)a^{n-2} - [ab]a^{n-1}\cdot
a \\
& \quad \quad \quad \quad + a[ab](n-1)a^{n-2}\cdot a \bigr\} \\
&= [ab]a^n -a[ab]na^{n-1}
\end{align*}

This generalizes readily to

\begin{equation*}
(-1)^{n+m}[a^nb^mab] = [ab]a^nb^m - a[ab]na^{n-1}b^m -
b[ab]a^nmb^{m-1}
\end{equation*}
and the result follows over 2FF.  The reduction to 1FF is immediate.
\end{proof}

Based on the preceding lemmas we have:

\begin{proposition}
Modulo 2FF, $\Phi$ must take the form:

\begin{equation}
\Phi(a,b) = 1 + [ab]\lambda(a,b) - a[ab]\partial_a\lambda(a,b) -
b[ab]\partial_b\lambda(a,b) \label{Phi2FF}
\end{equation}
with $\lambda(a,b)$ a commutative, symmetric power series, and
modulo 1FF, the form:

\begin{equation}
\Phi(a,b) = 1 + [ab]\lambda(a,b) \label{Phi1FF}
\end{equation}
with $\lambda(a,b)$ a commutative, symmetric power series.
\end{proposition}

\subsection{Solving the Hexagons Modulo 1FF}

\subsubsection{The Positive Hexagon}

We will gradually work the positive hexagon into a comparatively
simpler form, which can be solved.  Specifically, we will prove:

\begin{theorem}
Modulo 1FF, and under the assumptions (\ref{unitary}),
(\ref{nondegeneracy}) and (\ref{grouplike}) on $\Phi$, and the
assumption (\ref{rmatrix}) on $R$, the positive hexagon (\ref{hex})
becomes:

\begin{equation}
\label{hexplus} \lambda(a,b)+e^{-b}\ \lambda(b,c)+e^a\ \lambda(a,c)\
=\ \lbrace\frac{1}{ab}+\frac{e^{-b}}{bc}+\frac{e^a}{ac}\rbrace
\end{equation}
\label{thmhexplus}
\end{theorem}

\begin{remark}
Comparison with Kurlin's Compressed Hexagon \label{SignConvention}
\end{remark} This equation is equivalent to the `compressed'
equation obtained by Kurlin \cite{Kurlin}. However, if the equation
(\ref{hexplus}) is summarized as $L(\lambda) = R$, the equation
obtained by Kurlin is $L(\lambda) = -R$.  The sign is due to the
fact that Kurlin writes the hexagons (\ref{hex}) and (\ref{hexmin})
with the associators on the RHS appearing in the opposite order to
ours.  We have followed the convention used in \cite{Drinfel'd} and
\cite{DBN2}.

The first step in proving Theorem \ref{thmhexplus} is:

\begin{proposition}
Modulo 1FF, and under the assumptions (\ref{unitary}),
(\ref{nondegeneracy}) and (\ref{grouplike}) on $\Phi$, and the
assumption (\ref{rmatrix}) on $R$, the positive hexagon (\ref{hex})
becomes:
\begin{equation}
e^{b+c}\ -\ e^b \cdot e^c\ =\ [a,b]\lbrace \ \lambda(a,b)\ e^{b+c}\
+\ \lambda(b,c)\ e^c\ +\ \lambda(a,c)\ \rbrace \label{hex1FF}
\end{equation}
\end{proposition}

\begin{proof}
Plugging the expression (\ref{rmatrix}) for $R$ and (\ref{Phi1FF})
for $\Phi$ into the positive hexagon (\ref{hex}) we get:

\begin{align}
e^{b+c}\ = & \Bigl[1+[a,b]\lambda(a,b)\Bigr]\ \cdot\ e^b\ \cdot\
\Bigl[(1+[a,b]\ \lambda(a,b))^{-1}\Bigr]^{132} \\
& \quad \quad \quad \cdot\ e^c\ \cdot\ \Bigl[1+[a,b]\
\lambda(a,b)\Bigr]^{312} \notag
\end{align}

Now note that if the strands in a chord diagram are permuted
according to the transformation $(123)\ \rightarrow\ (132)$ the
chords $a,b,c$ transform as follows:
\begin{align*}
a\ \rightarrow\ c\\
b\ \rightarrow\ b\\
c\ \rightarrow\ a
\end{align*}

Applying this to (\ref{Phiinverse}) and looking modulo 1FF, we get:

\begin{equation*}
(\Phi^{-1})^{132} = 1 + [a,b]\ \lambda(b,c)
\end{equation*}

Similarly, under the strand permutation $(123)\ \rightarrow\ (312)$
the chords transform as:
\begin{align*}
a\ \rightarrow\ c\\
b\ \rightarrow\ a\\
c\ \rightarrow\ b
\end{align*}

Therefore,

\begin{equation*}
\Phi^{312} = 1 + [a,b]\ \lambda(c,a)
\end{equation*}

Accordingly, the hexagon becomes:
\begin{equation*}
e^{b+c}\ =\ (1+[a,b]\lambda(a,b))\cdot e^b \cdot
(1+[a,b]\lambda(b,c)) \cdot e^c\ (1+[a,b]\lambda(c,a))
\end{equation*}

We now multiply out on the RHS.  By the frozen feet property, we
need keep only the terms at most linear in [a,b].  Using also that
$\lambda$ is symmetric, we get

\begin{equation*}
e^{b+c}\ -\ e^b \cdot e^c\ =\ [a,b]\lbrace \ \lambda(a,b)\ e^be^c\
+\ \lambda(b,c)\ e^c\ +\ \lambda(a,c)\ \rbrace
\end{equation*}

The term $e^b e^c$ in the RHS becomes $e^{b+c}$, because it is
pre-multiplied by $[ab]$ and hence the $b$ and $c$ commute.
\end{proof}

We now want to find a better form for the expression $e^{b+c} - e^b
e^c$.  First, though, we make a comment concerning power series. In
light of the importance of the first `frozen foot' in any product,
we will find it useful to split power series in one or more
variables (such as the exponential function) into summands, each
with a different `foot'.  These feet do not commute with subsequent
variables, but the variables in the series which follow do commute
among each other. Thus for instance $\exp(x+y)$ will become

\begin{equation*}
e^{(x+y)} = 1 + x\cdot \frac{e^{(x+y)}-1}{x+y} + y\cdot
\frac{e^{(x+y)}-1}{x+y}
\end{equation*}

This notation will be very convenient, but it must be remembered
that, in general, the $x+y$ (or other variables) in the denominator
should only be treated as dividing (in this example) $(e^{(x+y)}
-1)$ and its summands, but not the feet $x$ and $y$. More generally,
in a term premultiplied by $a, b$ or $c$, a subsequent denominator
should only be treated as dividing factors that follow the foot, not
the foot itself. Put more technically, the feet act as a set of
generators for a module on which the commutative ring
$\Q[[a,a^{-1},b, b^{-1},c,c^{-1}]]$ acts by right multiplication. By
abuse of notation, we use the same letters to denote corrresponding
module and ring generators.

Further to this point, we note that the second equality in the
following would be incorrect:

\begin{equation*}
[xy] = xy - y \frac{xy}{y} = xy - \frac{y}{y} xy = xy - xy = 0
\end{equation*}

Instead, we have the rather trivial but still very useful lemma:

\begin{lemma}
For x, y any two different elements of \{a, b, c\},
\begin{equation}
x \frac{1}{x} - y \frac{1}{y} = [xy] \frac{1}{xy}
\end{equation}
\label{xxminusyy}
\end{lemma}

\begin{proof}
Proof is a simple calculation. \end{proof}

Using this lemma, we can derive the following proposition.

\begin{proposition}
\label{postriangle} Modulo 1FF, we have
\begin{equation}
e^{b+c}-e^b\cdot
e^c=[a,b]\cdot\bigl(\frac{e^{-a}}{ab}+\frac{e^c}{bc}+\frac{1}{ac}\bigr)
\end{equation}
\end{proposition}

\begin{proof}
\begin{align*}
e^{b+c}-e^b e^c & = \bigl( 1 + (b+c) \frac{e^{b+c}-1}{b+c} \bigr) -
\bigl( 1 + b \frac{e^b-1}{b} \bigr) \bigl( 1 + c \frac{e^c-1}{c}
\bigr) \\
& = (b+c) \frac{e^{b+c}-1}{b+c} - c \frac{e^c-1}{c} - b
\frac{e^b-1}{b} e^c \\
& = \bigl( c \frac{1}{c} - (b+c) \frac{1}{b+c} \bigr) + \bigl( (b+c)
\frac{1}{b+c} - b \frac{1}{b} \bigr) e^{b+c} - \bigl( c \frac{1}{c}
- b \frac{1}{b} \bigr) e^c \\
& = [c,b+c] \frac{1}{c(b+c)} + [b+c,b] \frac{1}{(b+c)b} e^{b+c} -
[c,b] \frac{1}{cb} e^c \\
& = [ab] \bigl(\frac{e^{-a}}{ab}+\frac{e^c}{bc}+\frac{1}{ac}\bigr)
\end{align*}
as required (note that in the last line, we have in particular used
Lemma \ref{eliminatec}(i) and (ii)).
\end{proof}

Putting together the last two propositions, dropping the $[a,b]$ on
both sides and multiplying by $e^a$, we get Theorem
\ref{thmhexplus}.

\begin{remark}
Singular Solution
\end{remark}
The simplified equation in Theorem \ref{thmhexplus} has the singular
solution $\lambda (x,y) = \frac{1}{xy}$.  We note in passing that
this is the factor which appears pre-multiplied by $[x,y]$ in the
difference $x \frac{1}{x} - y \frac{1}{y}$ (see Lemma
\ref{xxminusyy}). Non-singular solutions are discussed below.

\subsubsection{The Negative Hexagon}

It will be recalled that the associator is in fact required to solve
two hexagon equations, the positive and negative hexagons. The
positive hexagon was set out above and simplified. As mentioned
earlier, it can be shown (see \cite{Drinfel'd} at equation (2.10)
and \cite{DBN2} Prop. 3.7) that any solution of the positive hexagon
automatically satisfies the negative hexagon, provided it is
symmetric - in our notation, this means $\Phi \cdot \Phi^{321} = 1$,
or $\lambda(a,b)=\lambda(b,a)$ (which we have assumed of our $\Phi$
and $\lambda$).  In principle therefore, we need not concern
ourselves further with the negative hexagon.  However, it will be
useful at various points in the remainder of this paper to have a
statement of the negative hexagon as simplified modulo 1FF, and so
we set it out here.  The original version is:

\begin{equation}
\label{HexMinus}
(\Delta 1)R^{-1}\ =\ \Phi\ (R^{-1})^{23}\ (\Phi^{-1})^{132}\ (R^{-1})^{13}\ \Phi^{312}\\
\end{equation}
In terms of $a, b, c$ this means:
\begin{align}
e^{-(b+c)}\ =\ \Bigl[1+[a,b]\lambda(a,b)\Bigr]\ \cdot\ e^{-b}\
\cdot\ & \Bigl[(1+[a,b]\ \lambda(a,b))^{-1}\Bigr]^{132}\ \cdot\
e^{-c}\ \\
& \quad \cdot\ \Bigl[1+[a,b]\ \notag \lambda(a,b)\Bigr]^{312}
\end{align}

This in turn simplifies to:

\begin{equation}
\label{hexminus}
 \lambda(a,b)+e^{b}\ \lambda(b,c)+e^{-a}\
\lambda(a,c)\ =\ \frac{1}{ab}+\frac{e^{b}}{bc}+\frac{e^{-a}}{ac}
\end{equation}
by the same method as the positive hexagon.

\subsubsection{Solution to the Hexagons}

Subject to the same constraints (symmetry, non-degeneracy and
group-like form for $\Phi$, and exponential form for $R$), Kurlin
\cite{Kurlin} considered the image of the hexagon and pentagon
equations (in a slightly different form) under the map $Log :
\mathcal{GA} \longrightarrow \mathcal{L}$, where $\mathcal{GA}$
refers to the group-like elements of $\cA$ and $Log$ is given by the
usual power series. Kurlin considered the resulting equations modulo
the Lie ideal $[[\cL,\cL],[\cL,\cL]]$ (calling these equations the
`compressed' hexagons and pentagon). Kurlin found that the
compressed hexagons were equivalent to (\ref{hexplus}) and
(\ref{hexminus}) (up to a change of sign due to different
conventions -- see Remark \ref{SignConvention}), and proceeded to
derive a full set of solutions.

Theorem \ref{equivalence} below shows that solving the compressed
equations (i.e. $Log$ of the hexagons and pentagon, mod
$[[\cL,\cL],[\cL,\cL]]$) is equivalent to solving the hexagons and
pentagon modulo $1FF$. Thus Theorem \ref{thmhexplus} and the
corresponding derivation of (\ref{hexminus}) provide an alternative
proof that the compressed hexagons are equivalent to (\ref{hexplus})
and (\ref{hexminus}).

I note, out of interest, one of the solutions found by Kurlin:

\begin{equation}
\lambda(a,b)\ =\ \Bigl(\frac{\sinh (a+b)}{(a+b)}\ \cdot\
\frac{\omega}{\sinh \omega}\ -1\Bigr)\ /a\cdot b
\end{equation}

where $\omega\ =\ (a^2+ab+b^2)^{1/2}$.

Before stating and proving Theorem \ref{equivalence}, we need a few
lemmas.

\begin{lemma}
Let $\alpha + A$ be a Lie series in $\hat{\cL}$, with $\alpha$ the
linear part (i.e. of degree one).  Then, modulo 1FF:

\begin{equation*}
e^{\alpha + A} = e^{\alpha} + A \frac{e^{\alpha}-1}{\alpha}
\end{equation*}
\label{gplike1FF}
\end{lemma}

\begin{proof}
We consider the degree $n$ part of $e^{\alpha + A}$:
\begin{equation*}
\frac{1}{n!}(\alpha +A)^n = \frac{1}{n!}(\alpha^n + A\cdot
\alpha^{n-1}) \quad mod\ 1FF
\end{equation*}

Hence
\begin{align*}
e^{\alpha + A} &= 1 + (\alpha + A) +  \dotsb + 1/{n!} (\alpha^n + A\cdot \alpha^{n-1}) + \dotso \\
&= e^{\alpha} + A \bigl\{ 1 + \frac{1}{2} \alpha + \dotsb + \frac{1}{n!} a^{n-1} +
\dotso \\
&= e^{\alpha} + A\ \frac{e^{\alpha}-1}{\alpha}
\end{align*}
as required.
\end{proof}

\begin{proposition}
\label{EquivFFCC} Let $x+X,\ y+Y$ be Lie series in $\hat{\cL}$ with
$x$ and $y$ the linear parts.  Then the equality of group-like
elements

\begin{equation*}
e^{x+X} = e^{y+Y}
\end{equation*}
holds modulo 1FF if and only if the equality of their logarithmic
images

\begin{equation*}
x+X = y+Y
\end{equation*}
holds modulo $[[\cL,\cL],[\cL,\cL]]$. \label{prelimequiv}
\end{proposition}

\begin{proof}
We have
\begin{equation*}
e^{x+X} = e^{y+Y}
\end{equation*}
if and only if
\begin{equation*}
1+x+X + \text{h.o.} = 1 + y + Y + \text{h.o.}
\end{equation*}
(where `h.o.' means `higher order terms') and hence, by looking at
degree one terms, we must have $x=y$.

Then
\begin{gather*}
e^{x+X} = e^{x+Y} \quad mod\ 1FF \\
\iff \\
e^x + X \frac{e^x-1}{x} = e^x + Y \frac{e^x-1}{x} \quad mod\ 1FF \\
\iff \\
X = Y \quad mod\ 1FF \\
\iff \\
 X = Y \quad mod\ [[\cL,\cL],[\cL,\cL]]
\end{gather*}
as required.  Note that in going from the second to the third line,
we used the summation

\begin{equation*}
\frac{e^x -1}{x} = 1 + \frac{1}{2} x + \frac{1}{3!} x^2 + \dotso
\end{equation*}
which shows that $\frac{e^x -1}{x}$ is invertible, and in going from
the third to the fourth lines, we used Proposition \ref{lieseries}.
\end{proof}

\begin{theorem}
The hexagons and the pentagon hold modulo 1FF if and only if their
logarithmic images hold modulo $[[\cL,\cL],[\cL,\cL]]$.
\label{equivalence}
\end{theorem}

\begin{proof}
The LHS and RHS of the hexagons and pentagon are (products of)
group-like elements, hence are themselves group-like elements. We
can now apply Proposition \ref{prelimequiv}.
\end{proof}

\begin{remark}
Campbell-Hausdorff-Baker Formula, Modulo $\FF$ or
$[[\cL,\cL],[\cL,\cL]]$
\end{remark}

Lemma \ref{gplike1FF} can be used to give a short, simple derivation
of Kurlin's formula for the Campbell-Hausdorff-Baker formula modulo
$[[\cL,\cL],[\cL,\cL]]$ (see \cite{Kurlin} at Prop. 2.8 and Prop.
2.12), which also holds modulo $\FF$. Indeed, we seek a power series
$C(x,y)$ in commutative variables $x,y$ such that

\begin{equation*}
\exp(b+c+[bc]C(b,c)) = \exp(c) \exp(b)
\end{equation*}

But from Lemma \ref{gplike1FF}, we have

\begin{equation*}
e^{(b+c+[bc]C(b,c))} = e^{(b+c)} + [bc]C(b,c)\ \frac{e^{(b+c)} -
1}{(b+c)}
\end{equation*}
(where technically we should perhaps have converted $[bc]C(b,c)$ to
the Lie series $[C(-b,-c)bc]$ using Proposition \ref{liemonomials}
before applying Lemma \ref{gplike1FF}, and then converted back to
$[bc]C(b,c)$).  But then

\begin{align*}
[bc]C(b,c)\ \frac{e^{(b+c)} - 1}{(b+c)} &=
e^c e^b - e^{(b+c)} \\
&= -[ac] \bigl(\frac{e^{-a}}{ac}+\frac{e^b}{bc}+\frac{1}{ab}\bigr)
\end{align*}
from Proposition \ref{postriangle} (after the exchange $b
\leftrightarrow c$). From this we readily derive:

\begin{equation*}
C(b,c) = \frac{e^b-1}{bc}\ \bigl( \frac{b+c}{e^{b+c}-1} -
\frac{b}{e^b-1} \bigr) \end{equation*}

This result is valid modulo $\FF$ and, pursuant to Proposition
\ref{EquivFFCC}, modulo $\Lfour$.

\subsubsection{Solving the Pentagon Modulo 1FF}

It so happens that the pentagon is automatically satisfied modulo
1FF, for any function $\Psi$ of the form $\Psi = [ab]\ \mu(a,b)$,
where $\mu$ is a commutative power series of two variables which is
symmetric in its arguments.  In particular, the 1FF associator
$\Phi$ satisfies the pentagon since it has the required symmetry
property (and this, independent of the particular form $\Phi$ must
take to solve the hexagons).  Since the proof already appears in
\cite{Kurlin} Proposition 5.10, in this section we merely present
notation, and state and prove results in the form that will be used
in the balance of this paper.  This material largely reproduces
results by Kurlin but the results and proofs are given here in a
somewhat more general and concise form.

We will find expressions for $(\Delta 11) \Psi$, $(1 \Delta 1) \Psi$
and $(11 \Delta) \Psi$ as sums of terms of the form
$[x,y]\lambda(u,v)$, where $x,\ y,\ u,\ v$ are all single-chord
diagrams on four strands. In fact, each of $(\Delta 11) \Psi$, $(1
\Delta 1) \Psi$ and $(11 \Delta) \Psi$ will be a sum of 4 terms of
that form. Then, as shown by Kurlin \cite{Kurlin} in his Proposition
5.10, it can be seen that cancellations occur in the (linearized)
pentagon, leading to the desired equality.

Since the pentagon lives in the space of chord diagrams on four
strands, we need to make precise what this space is. Thus we will
will be concerned with the space $\cA_4$ (or rather its completion
$\hat{\cA_4}$) generated by single horizontal chord diagrams on 4
vertical strands, ie

\begin{align}
a:= \ta, \quad & b:= \tb, \quad c:= \tc \\
d:= \td, \quad & e:= \te, \quad f:= \tf \notag \label{chords}
\end{align}
where $\tij$ represents the chord diagram with a single horizontal
chord resting on strands $i$ and $j$.  Obviously, $\tij = t^{ji}$.

For each $l=1, \dotso, 4$, we get the expected $4T$ relations among
the chord diagrams whose endpoints rest on the three strands other
than strand $l$, ie

\begin{equation}
[t^{ij},t^{jk}] = [t^{jk},t^{ki}] = [t^{ki},t^{ij}] \label{4Tijk}
\end{equation}
where $l\notin \{i,j,k\}$.

These are just the $4T$ relations we would get if we dropped all
chord diagrams with chords resting on strand $l$, and viewed the
remaining chord diagrams as forming a copy of $\cA_3$.  In addition,
however, we have a new kind of relation, known as `locality in
space' relations, which provide that we can commute any two chords
whose four endpoints rest on four different strands. In other words,

\begin{equation}
[\tij,\tkl] = 0, \quad whenever \ \#\{i,j,k,l\}=4 \label{locality}
\end{equation}

We will use the notation $\AFF$ to refer to $\hat{\cA}_4$ modulo
1FF. Of course, 1FF in this context refers to the relations

\begin{equation*}
u[xy] = 0 \quad for \ u,x,y \in \cA_4, deg(u) \geq 1
\end{equation*}

One can easily check that the element $a+b+c+d+e+f = \sum_{1\leq i <
j \leq 4} \tij$ is central in $\hat{\cA}_4$.  Since $a+b+c$ is
central in the algebra $\hat{\cA}_3$ generated by $\{a,b,c\}$, it
follows also that

\begin{equation}
[d+e+f,[ab]] = 0 \label{def}
\end{equation}
and hence

\begin{equation*}
[ab]f = [ab](-e-d)
\end{equation*}
modulo 1FF.

We will now also take $\cL$ to be the Lie algebra generated by the
symbols $\{\tij\}_{1\leq i<j \leq 4}$, with the commutation
relations given by $4T$ and the locality relations. In fact we
really deal with the completion $\hat{\cL}$, but will still write
$\cL$. Moreover, we take $\LCC$ to refer to (the completed) $\cL$
modulo $\Lfour$.

Finally we introduce the notation:
\begin{itemize}
\item $\overline{\tij}:= t^{4k} \quad i \negthinspace f\ \{i,j,k\}=\{1,2,3\}$
\item $\overline{t^{4i}}:= t^{4i} \quad i \negthinspace f\ i\in\{1,2,3\}$
\end{itemize}
Thus, `barring' a chord amounts to replacing it with its
`complementary' chord (ie, resting on complementary strands), with
which it commutes, if the chord does not sit on strand 4 (ie,
$\bar{a}=e,\ \bar{b}=f,\ \bar{c}=d$).  Barring a chord that sits on
strand $4$ has no effect (ie $\bar{d}=d,\ \bar{e}=e,\ \bar{f}=f$).

We will also use a short-hand notation for commutators:  if $l\in
\{1,2,3,4\}$, and $(i,j,k) = (1,...,\hat{l},...,4)$, we write:

\begin{equation*}
[\hat{l}] := [\tij,\tjk]
\end{equation*}
ie we identify a (degree two) commutator by the strand it does not
touch, with sign given by the stated assumption on the order of
$i,j,k$.

We now derive some basic results about the Lie algebra $\LCC$. We
first give a set of vector space generators of $\LCC$ (compare
\cite{Kurlin} Lemma 5.5).

\begin{proposition}
$\LCC$ is generated as a vector space by the elements:
\begin{enumerate}
\item $[(\tij)^r(\tjk)^s[\hat{l}]]$, where $(i,j,k)=(1, ...,
\hat{l}, ..., 4)$, and $r,s \in \N$. These just generate the Lie
subalgebras obtained by dropping strand $l$, and viewing diagrams on
the remaining strands as constituting a copy of $\hat{\cL}_3$; and

\item $[(t^{i4})^r(t^{j4})^s[\hat{4}]]$, where $\{i,j\} \subseteq
\{1, ..., \hat{l}, ..., 4\}$, and $r,s \in \N$.
\end{enumerate}
\label{L4generators}
\end{proposition}

To prove Proposition \ref{L4generators}, we need three lemmas,
starting with this lemma which applies in $\cL_4$ (compare
\cite{Kurlin} Claim 5.2(c)):

\begin{lemma}
In $\cL_4$ we have
\begin{equation*}
[t^{rs}[\hat{s}]]=(-1)^s [\overline{t^{rs}}[\hat{4}]], \quad r,s
\text{ distinct elements of } \{1,2,3,4\}
\end{equation*}
\label{convertkto4}
\end{lemma}

\begin{proof}

We let $i\in \{1,2,3\}$, and then choose $j,k$ so that $(ijk)$ is a
cyclic permutation of $(123)$.  Thus $[\tij \tjk]=[\hat{4}]$.

We note that, by 4T, $[\tij \tjk] = [\tjk \tki]$. Hence, $[\ti \tij
\tjk]$ can also be written $[\ti \tjk \tki]$, and:

\begin{align*}
[\ti \tij \tjk] &= -[\tij \tjk \ti] - [\tjk \ti \tij] = -[\tjk
\ti \tij] \\
[\ti \tjk \tki] &= -[\tjk \tki \ti] - [\tki \ti \tjk] = -[\tjk \tki
\ti]
\end{align*}
where the first equality in each line is the Jacobi relation, and
the second equality comes from the locality in space relations.

From the assumptions on $i,j,k$, we see that $(4,i,j)$ is some
permutation of $(1,...,\hat{k},...4)$, and one can readily confirm
that the permutation is in fact \emph{cyclic} iff k is odd.  Hence
we get:

\begin{equation}
[\tjk [\hat{k}]] = (-1)^k [\ti \tij \tjk] = (-1)^k
[\overline{\tjk}[\hat{4}]] \label{khatITO4hat}
\end{equation}

By similar reasoning one can see that:

\begin{equation}
[\tjk [\hat{j}]] = (-1)^j [\ti \tjk \tki] = (-1)^j
[\overline{\tjk}[\hat{4}]] \label{jhatITO4hat}
\end{equation}

Next we can repeat the process to get:

\begin{equation*}
[\ti[\tij \tjk]] = -[\tjk \ti \tij] = -[\tjk t^{j4} \ti] = [t^{j4}
\ti \tjk] + [\ti \tjk t^{j4}] = [\ti \tjk t^{j4}]
\end{equation*}
Here, for the first equality we used the $4T$ relations to get
$[t^{4i} \tij] = [t^{j4} \ti]$, for the second equality we used
Jacobi, and for the third we used the locality in space relations.

We now note that $[\tjk \tj] = (-1)^{\eta} [\hat{i}]$, where $\eta =
+1$ (or $-1$) when $(k,j,4)$ is (or is not) a cyclic permutation of
$(1, ...,\hat{i},... 4)$.  Moreover, $(k,j,4)$ is a cyclic
permutation of $(1, ..., \hat{i}, ..., 4)$ if and only if $i$ is
even. Hence:

\begin{equation*}
[\ti [\hat{i}]] = (-1)^i [\ti \tij \tjk] = (-1)^i [\ti[\hat{4}]]
\end{equation*}

When $s \neq 4$, Equation (\ref{khatITO4hat}) gives us the desired
result in the case $r=j<k=s$, Equation (\ref{jhatITO4hat}) in the
case $s = j<k=r$, and the last equation in the case $r=4 \neq s$.
The case $s=4$ is trivial.

\end{proof}

We now move to our various quotient spaces.  We first note the
remark (see also \cite{Kurlin} Lemma 5.4(a)):

\begin{remark}
By an obvious generalization of the proof of Proposition
\ref{lieseries}, if w is a word on $\{a,b,c,d,e,f\}$ with at least
two letters, and $u,v \in \{a,b,c,d,e,f\}$, we have

\begin{align*}
[uvw] &= [vuw] \quad mod\ [[\cL,\cL],[\cL,\cL]] \\
[w]uv &= [w]vu \quad mod\ 1FF\ or\ 2FF
\end{align*}

\end{remark}

In $\LCC$ we get the result (compare \cite{Kurlin} Lemma 5.4(b)):

\begin{lemma}
In $\LCC$, if $\omega$ is a commutative power series in the $\tij,\
1\leq i<j \leq 4$, and $k=1,2,3$, we have

\begin{equation*}
[\omega\ \tk [\hat{4}]] = [\overline{\omega}\ \tk [\hat{4}]]
\end{equation*}
ie the presence of the factor $\tk$ multiplying the commutator
$[\hat{4}]$ allows us to replace all $\tij$s by $\overline{\tij}$s.
\label{converttobar}
\end{lemma}

\begin{proof}
It suffices to show this for monomials of the given form, and also
we may assume $\{i,j,k\} = \{1,2,3\}$ since the result is trivial
when $i$ or $j$ is equal to 4.  We first consider the case $[\tij
\tj [\hat{4}]]$:

\begin{equation*}
[\tij \tj [\hat{4}]] = (-1)^j [\tij \tj [\hat{j}]] = (-1)^j [\tj
\tij [\hat{j}]] = [\tj \tk [\hat{4}]] = [\tk \tj [\hat{4}]]
\end{equation*}
(where the third equality uses Lemma {convertkto4}) as needed, for
this case.

Next we take the case $[\tij \tk [\hat{4}]]$.  We note that $[\tij
[\hat{4}]] = [(-\tjk - \tki) [\hat{4}]]$ by 4T, so:

\begin{align*}
[\tij \tk [\hat{4}]] &= [\tk \tij [\hat{4}]] = [\tk (-\tjk - \tki)
[\hat{4}]] \\
&= [\tk (-\ti -\tj)[\hat{4}]] = [\tk \tk [\hat{4}]]
\end{align*}
where in the second last equality we applied the first case twice,
and in the last equality we applied Equation (\ref{def}). This
completes the proof.
\end{proof}

We now give a lemma, valid in $\LCC$, which in particular will allow
us to derive explicit expressions for the action of strand doubling
in $\LCC$ (see \cite{Kurlin} Claim 5.7):

\begin{lemma}
We take $l\in \{1,2,3,4\}$, and $(i,j,k)$ a cyclic permutation of
$(1, ..., \hat{l}, ... , 4)$ (so that $[\tij,\tjk] = [\hat{l}]$). We
also take $u\in \{i,j,k\}$ and $m$ a non-negative integer. Then, in
$\LCC$,
\begin{equation*}
[(\tij + t^{ul})^m [\hat{l}]] = [(\tij)^m \hl] -
(-1)^l[(\overline{\tij})^m [\hat{4}]] + (-1)^l [(\overline{\tij} +
\overline{t^{ul}})^m [\hat{4}]]
\end{equation*}
\label{expandDelta}
\end{lemma}

\begin{proof}
\begin{align*}
[(\tij + t^{ul})^m \hl]  &= [(\tij)^m  \hl]  + \sum_{r=1}^m {m
\choose r} [(\tij)^{m-r} (t^{ul})^r  \hl]  \\
&= [(\tij)^m   \hl] + (-1)^l \sum_{r=1}^m {m \choose
r} [(\tij)^{m-r} (\overline{t^{ul}})^r   [\hat{4}]] \\
&= [(\tij)^m   \hl] + (-1)^l \sum_{r=1}^m {m \choose
r} [(\overline{\tij})^{m-r} (\overline{t^{ul}})^r   [\hat{4}]] \\
&= [(\tij)^m   \hl] - (-1)^l[(\overline{\tij})^m   [\hat{4}]] +
(-1)^l [(\overline{\tij}+ \overline{t^{ul}})^m  [\hat{4}]]
\end{align*}
as required (where in going from the first to second lines we used
Lemma \ref{convertkto4}, and in going from the second to third lines
Lemma \ref{converttobar} applies to $\tij$ since, in the summation,
$\overline{t^{ul}}$ appears with the power $r \geq 1$).

\end{proof}

\begin{proof}[Proof of Proposition \ref{L4generators}]
The proof of Proposition \ref{L4generators} is now a straightforward
combination of the past three lemmas.
\end{proof}

We can now derive explicit expressions for the action of $\Delta
11,\ 11\Delta\ and\ 1\Delta 1$ on $\Psi(\ta,\tb) = [\ta
\tb]\lambda(\ta,\tb)$ (\cite{Kurlin} Lemma 5.8).

\begin{proposition}
Mod $[[\cL,\cL],[\cL,\cL]]$, we have:
\begin{align*}
(\Delta 11)\Psi &= [\lambda(\tc,\te)  [\hat{2}]]  +
[\lambda(\tb,\te)   [\hat{1}]] -
[\lambda(\td,\te)  [\hat{4}]]  + [\lambda(\tf,\te)  [\hat{4}]]  \\
(11 \Delta) \Psi &= [\lambda(\ta,\td)  [\hat{3}]]  +
[\lambda(\ta,\tb)  [\hat{4}]]  - [\lambda(\te,\tf)  [\hat{4}]]  + [\lambda(\te,\td)  [\hat{4}]]  \\
(1 \Delta 1) \Psi &= [\lambda(\ta,\td)   [\hat{3}]] +
[\lambda(\tc,\te)  [\hat{2}]]  + [\lambda(\te,\td)  [\hat{4}]]  -
[\lambda(\td,\te) [\hat{4}]]
\end{align*}
\label{D11expansion}
\end{proposition}

\begin{proof}
It suffices to prove the statements for monomials.  We begin with:

\begin{align*}
(\Delta 11) [(\ta)^k (\tb)^l  [\hat{4}]]  &= [(\tc+\tb)^k
(\te)^l [\tc+\tb,\te]] \\
&= [(\tc+\tb)^k (\te)^l  [\hat{2}]]  + [(\tc+\tb)^k (\te)^l  [\hat{1}]]  \\
&= [(\tc)^k (\te)^l  [\hat{2}]]  - [(\td)^k (\te)^l  [\hat{4}]]  +
[(\td+\tf)^k (\te)^l  [\hat{4}]]  \\
& \quad + [(\tb)^k (\te)^l   [\hat{1}]] + [(\tf)^k (\te)^l
[\hat{4}]] - [(\td+\tf)^k (\te)^l  [\hat{4}]]  \\
&= [(\tc)^k (\te)^l [\hat{2}]]  + [(\tb)^k (\te)^l  [\hat{1}]]  -
[(\td)^k (\te)^l  [\hat{4}]]  \\
& \quad + [(\tf)^k (\te)^l  [\hat{4}]]
\end{align*}
where we used Lemma \ref{expandDelta} in going from the second to
the third lines.

The proof of the relation for $(11 \Delta)$ and $(1 \Delta 1)$ is
similar (though involving iterated use of Lemma \ref{expandDelta} in
the case of $(1 \Delta 1)$).
\end{proof}

We note finally that, while the results of this section have been
stated in terms of the Lie algebras $\cL$ and $\LCC$, by means of
Proposition \ref{liemonomials} these results have immediate
analogues in the language of $\cA$ and $\AFF$.

\subsection{Solving the Positive Hexagon Modulo 2FF}

\subsubsection{Overview}

We follow the same general strategy for reworking the positive
hexagon into a usable form modulo 2FF as we used modulo 1FF.
However, the details are more involved, and so an overview of the
specifics may be useful.

We are taking $\Phi$ to be of the form

\begin{equation}
\Phi(a,b) = 1 + [ab]\lambda(a,b) - a[ab]\partial_a\lambda(a,b) -
b[ab]\partial_b\lambda(a,b) \label{Phi2FFrepeat}
\end{equation}
where $\lambda(a,b)$ is a symmetric power series in the commutative
variables $a$ and $b$ (see equation (\ref{Phi2FF})).

We plug this expression into the positive hexagon, and rework it
into the form

\begin{equation*}
e^{b+c} - e^be^c =\ \{\text{expression in }\lambda \text{ and its
partials} \}
\end{equation*}
where we refer to the LHS as the `triangle'.  We find that the RHS
is a sum of terms premultiplied by $[ab]$, $b[ab]$ and $c[ab]$.

We then find an expression for the triangle which, as it happens,
consists also of a sum of terms premultiplied by $[ab]$, $b[ab]$ and
$c[ab]$.

Comparing the $[ab]$ terms on the LHS and RHS, we find an equation
in $\lambda$ only (no partials) which is just the 1FF positive
hexagon.  We then compare the $b[ab]$ terms on the LHS and RHS, and
find that the result is just the operator $(1 + \partial_a -
\partial_b)$ applied to the 1FF positive hexagon.  Similarly we find
that the $c[ab]$ terms simply give us the operator $(1 + \partial_a
- \partial_c)$ applied to the 1FF positive hexagon.  Not
surprisingly, it is again true that any $\lambda$ which satisfies
the positive hexagon also automatically satisfies the negative
hexagon.

We conclude that equation (\ref{Phi2FFrepeat}) gives a solution to
the 2FF hexagon whenever $\lambda$ is a solution to the 1FF hexagon.
Moreover, since by factoring out 2FF we get a quotient of Kurlin's
quotient, Kurlin's argument still applies to show that the unitarity
condition implies that such a solution is also a solution to the
pentagon.

\subsubsection{Simplifying the Positive Hexagon}

We go back to the equation (\ref{Phi2FFrepeat}) giving $\Phi$ in the
form:

\begin{equation*}
\Phi(a,b) = 1 + [ab]\lambda(a,b) - a[ab]\partial_a\lambda(a,b) -
b[ab]\partial_b\lambda(a,b)
\end{equation*}
with $\lambda(a,b)$ a commutative, symmetric power series.  Note
that this expression involves only the two variables $a$ and $b$,
but not $c$. However, the positive hexagon also involves expressions
in $(\Phi^{-1})^{132}$ and $\Phi^{312}$, and after the corresponding
permutations are given effect, the resulting expressions will
involve $a,\ b\ and\ c$.  We would like to get rid of one of these
variables, for instance by replacing $c$ by $-a-b$.  Lemma
\ref{eliminatec} tells us we can go this whenever $c$ appears after
a commutator, providing we are proceeding modulo 1FF.  However this
is no longer true modulo 2FF.  Instead, we have the following lemma
which is valid in $\hat{\cA}$, not just in a quotient:

\begin{lemma}

Given any power series $\alpha(a,b,c) \in \cA_3$,

\begin{equation}
[a,b]\ \alpha(a,b,c)=[ab]\ \alpha(a,b,-a-b) + (a+b+c)[ab]\
\partial_c \alpha(a,b,c) \label{[ab]alpha}
\end{equation}
where the partial is evaluated at $c=-a-b$. \label{corpartials}
\end{lemma}

\begin{proof}
Using the fact that $a+b+c$ is central, we have
\begin{equation*}
(a+b+c)[ab]=[ab](a+b+c)
\end{equation*}
hence
\begin{equation*}
[ab]c=[ab](-a-b) +(a+b+c)[ab]
\end{equation*}
hence
\begin{equation*}
[ab]c^2=[ab](-a-b)^2 + 2(a+b+c)[ab](-a-b)
\end{equation*}
and more generally
\begin{equation*}
[ab]c^n=[ab](-a-b)^n + n(a+b+c)[ab](-a-b)^{n-1}
\end{equation*}
and indeed
\begin{equation}
[a,b]\ \alpha(a,b,c)=[ab]\ \alpha(a,b,-a-b) + (a+b+c)[ab]\
\partial_c \alpha(a,b,c) \label{[ab]alpha}
\end{equation}
where the partial is evaluated at $c=-a-b$.
\end{proof}

Note that in practice we will actually use this lemma to replace $a$
by $-b-c$, as we will later find it more convenient to work with $b$
and $c$.

Hence we will actually take $\Phi$ to have the form:
\begin{equation*}
\Phi(a,b) = 1 + [ab]\lambda(-b-c,b) - b[ab](\ly - \lx) - c[ab]\lx
\end{equation*}
where the partials are evaluated at $(x,y) = (-b-c,b)$.  In
practice, though, we will write this as

\begin{equation*}
\Phi(a,b) = 1 + [ab]\lambda(a,b) - b[ab](\lb - \la) - c[ab]\la
\end{equation*}
remembering that in fact $a=-b-c$ in expressions premultiplied by a
commutator.

We now need to find expressions for $(\Phi^{-1})^{132}$ and
$\Phi^{312}$.  We already have an expression for $\Phi^{-1}$, namely
equation (\ref{Phiinverse}). Replacing $w$ by $\lambda$ and its
partials, this becomes:

\begin{equation*}
\Phi(a,b)^{-1} = 1 - [ab]\lambda(a,b) + a[ab]\la + b[ab]\lb
\end{equation*}

Also, as indicated earlier, under the permutation $(123) \rightarrow
(132)$, chords get permuted as follows:

\begin{align}
a & \rightarrow c \notag \\
b & \rightarrow b \notag \\
c & \rightarrow a \notag
\end{align}

Hence

\begin{align}
(\Phi^{-1})^{132} & = 1 - [cb]\lambda(c,b) + c[cb]\lc +
b[cb]\partial_b\lambda(b,c) \notag \\
& = 1+ [ab]\lambda(b,c) - b[ab]\partial_b\lambda(b,c) - c[ab]\lc
\notag
\end{align}
where we have used the symmetry of $\lambda$ and the relation
$-[cb]=[ab]$.

Under the permutation $(123) \rightarrow (312)$, the chords are
permuted as

\begin{align}
a & \rightarrow c \notag \\
b & \rightarrow a \notag \\
c & \rightarrow b \notag
\end{align}

Hence

\begin{align*}
\Phi^{312} & = 1 + [ab]\lambda(c,a) - c[ab]\partial_c\lambda(c,a) -
a[ab]\partial_a\lambda(a,c) \\
&= 1 + [ab]\lambda(c,a) + b[ab]\partial_a\lambda(a,c) -
c[ab](\partial_c\lambda(c,a) - \partial_a\lambda(a,c))
\end{align*}
where we have used the relation $[ca]=[ab]$.

We can now write the positive hexagon equation modulo 2FF:

\begin{align*}
e^{b+c} = & \bigl( 1+ [ab]\lambda(a,b) - b[ab](\lb - \la) +
c[ab]\la \bigr) \cdot e^b \\
& \cdot \bigl( 1+ [ab]\lambda(c,b) - b[ab]\partial_b\lambda(b,c) -
c[ab]\partial_c\lambda(c,b)
\bigr) \cdot \\
& e^c \cdot \bigl( 1 +[ab]\lambda(c,a) + b[ab]\partial_a\lambda(a,c)
- c[ab](\partial_c\lambda(c,a) - \partial_a\lambda(a,c)) \bigr)
\end{align*}

Linearization still holds in 2FF, so we need to keep only terms that
are up to linear in [ab], and we get

\begin{align}
\label{mess3} e^{b+c} - e^b e^c &= [ab]\lambda(a,b) e^{b+c} + e^b
[ab]
\lambda(b,c) e^c + e^b e^c [ab]\lambda(c,a) \\
& \quad -b[ab](\partial_b\lambda(b,a) - \partial_a\lambda(a,b))
e^{b+c} -e^b b[ab]\partial_b\lambda(b,c) e^c \notag \\
& + e^be^cb[ab]\partial_a\lambda(a,c) + c[ab]\partial_a\lambda(a,b)
e^{b+c} - e^b c[ab] \partial_c\lambda(c,b) e^c \notag \\
& - e^be^c c[ab](\partial_c\lambda(c,a)-\partial_a\lambda(a,c)) \notag \\
& = [ab]\bigl\{ \lambda(a,b) e^{b+c} + \lambda(b,c)e^c +
\lambda(c,a) \bigr\} \\
&+ b[ab] \bigl\{ \lambda(b,c)e^c + \lambda(c,a) -
(\partial_b\lambda(b,a) - \partial_a\lambda(a,b)) e^{b+c} \notag \\
& - \partial_b\lambda(b,c)e^c + \partial_a\lambda(a,c) \bigr\} \notag \\
& \quad + c[ab] \bigl\{ \lambda(c,a) + \partial_a\lambda(a,b)
e^{b+c} - \partial_c\lambda(c,b)e^c - \partial_c\lambda(c,a) +
\partial_a\lambda(a,c) \bigr\} \notag
\end{align}
where, again, we have throughout $a=-b-c$ whenever $a$ is
pre-multiplied by $[ab]$.

\subsubsection{Simplifying the Triangle Relation}

As with the 1FF case, we need to determine $e^{b+c}-e^b e^c$, this
time modulo 2FF.  The procedure is analogous to the 1FF case.  First
we note the following expression for the exponential function, which
is useful modulo 2FF:

\begin{equation*}
e^x = 1 + x + x^2\ \frac{e^x -x- 1}{x^2}
\end{equation*}

We will prove the following proposition concerning the difference
$e^{b+c} - e^b e^c$:

\begin{proposition}
Modulo 2FF, the triangle difference can be expressed:
\begin{align}
e^{b+c} - e^b e^c &= [a,b]\ \Bigl( \frac{e^{b+c}}{ab}
+ \frac{e^c}{bc} + \frac{1}{ac} \Bigr) \\
& +b[a,b] \Bigl( \frac{ e^{b+c}}{a^2 c} -
\frac{e^{b+c}}{b^2c} + \frac{e^c}{b^2c} + \frac{e^c}{bc} + \frac{1}{ac} - \frac{1}{a^2c} \Bigr) \notag \\
& +c[a,b] \Bigl( -\frac{e^{b+c}}{a^2 b} + \frac{e^c}{bc^2} +
\frac{1}{ac} - \frac{1}{bc^2} + \frac{1}{a^2b} \notag \\
& \quad \quad \quad - \frac{1}{ab} + \frac{1}{a^2b} - \frac{1}{bc} -
\frac{1}{bc^2} \Bigr) \notag
\end{align} \label{triangle2FF}
\end{proposition}

The proof relies on a lemma which is the equally trivial, but
equally useful 2FF analogue of Lemma \ref{xxminusyy}:

\begin{lemma}
For x and y any two different elements of \{a, b, c\}, we have the
following modulo 2FF:
\begin{equation*} x^2 \frac{1}{x^2} - y^2 \frac{1}{y^2} = [xy]
\frac{1}{xy} + x[xy] \frac{1}{x^2y} + y[xy] \frac{1}{xy^2}
\end{equation*}
There are a couple of variants of this equation which will also be
useful:
\begin{align*}
x^2 \frac{1}{x} - xy \frac{1}{y} & = x[xy] \frac{1}{xy} \\
-(x+y)^2 \frac{1}{x+y} + xy \frac{1}{y} + y^2 \frac{1}{y} &= x[yx]
\frac{1}{(x+y)y} + y[yx] \frac{1}{(x+y)y}
\end{align*}
\end{lemma}

\begin{proof}
The proofs are straightforward if tedious and are omitted.
\end{proof}

\begin{proof}[Proof of Proposition~\ref{triangle2FF}]
We plug our expression for the exponential function mod 2FF into the
triangle $e^{b+c} - e^b e^c$ and get:

\begin{align*}
e^{b+c} - e^b e^c &= \Bigl( 1+ (b+c) + (b+c)^2\ \frac{e^{b+c} -
(b+c) -1}{(b+c)^2} \Bigr) \\
& \quad - \Bigl(1 + b + b^2\ \frac{e^b-b-1}{b^2} \Bigr) - \Bigl(
1+c+c^2\ \frac{e^c-c-1}{c^2} \Bigr) \\
&= (b+c)^2\ \frac{e^{b+c} - (b+c) -1}{(b+c)^2} - b^2\
\frac{e^b-b-1}{b^2}\ e^c \\
& \quad -b-bc\ \frac{e^c-1}{c} -1-c- c^2\ \frac{e^c-c-1}{c^2} \\
&= (b+c)^2\ \frac{e^{b+c}}{(b+c)^2} - b^2\ \frac{e^{b+c}}{b^2} \\
& \quad + b^2\ \frac{e^c}{b^2} - c^2\ \frac{e^c}{c^2} \\
& \quad + b^2\ \frac{be^c}{b^2} - bc\ \frac{e^c}{c} \\
& \quad -(b+c)^2 \frac{1}{b+c} + bc \frac{1}{c} + c^2\frac{1}{c} \\
& \quad - (b+c)^2 \frac{1}{(b+c)^2} + c^2\frac{1}{c^2}
\end{align*}

From here we use various substitutions from the previous lemma.  The
result then follows from a few additional simple manipulations,
which are omitted.
\end{proof}

\subsubsection{Solution Modulo 2FF}

We can now bring together the different components of the hexagon,
namely the expression for $e^{b+c} - e^be^c$ and the expression
involving the $\lambda$'s (see equation (\ref{mess3}).  We get terms
that are pre-multiplied by $[ab]$, $b[ab]$ and $c[ab]$:

\paragraph{Terms in $[ab]$}

\begin{equation}
\frac{e^{b+c}}{ab} + \frac{e^c}{bc} + \frac{1}{ac} =
\lambda(a,b)e^{b+c} + \lambda(c,b)e^c + \lambda(c,a) \notag
\end{equation}

Note that this is just the (positive) hexagon equation in 1FF. In
the following we will refer to the LHS as $Hex^l$ and to the RHS as
$Hex^r$.

\paragraph{Terms in $b[ab]$}

We have:

\begin{align*}
\frac{ e^{b+c}}{a^2 c} - \frac{e^{b+c}}{b^2c} + \frac{e^c}{b^2c} +
\frac{e^c}{bc} + \frac{1}{ac} - \frac{1}{a^2c} &=
\partial_a\lambda(a,b) e^{b+c} - \partial_b\lambda(a,b) e^{b+c} \\
&+ \lambda(c,b) e^c -
\partial_b\lambda(b,c) e^c + \lambda(c,a) +
\partial_a\lambda(a,c)
\end{align*}

I now claim that this $b[ab]$ equation is simply the statement

\begin{equation}
\bigl[ 1+\partial_a-\partial_b \bigr] Hex^l = \bigl[ 1+\partial_a-
\partial_b \bigr] Hex^r \notag
\end{equation}
and hence follows from the hexagon. On the RHS, this is easy to
check:

\begin{align}
\bigl[ 1+\partial_a- \partial_b \bigr] Hex^r = & \lambda(a,b)e^{b+c}
+ \lambda(c,b)e^c + \lambda(c,a) \\
& + (\partial_a\lambda(a,b))e^{b+c} + \partial_a\lambda(a,c) \notag
\\
& - (\partial_b\lambda(a,b))e^{b+c} - \lambda(a,b)e^{b+c}
-(\partial_b\lambda(c,b))e^c \notag \\
& = \partial_a\lambda(a,b) e^{b+c} - (\partial_b\lambda(a,b))
e^{b+c} + \lambda(c,b) e^c - (\partial_b\lambda(b,c)) e^c \notag \\
& \quad \quad \quad \quad + \lambda(c,a) +
\partial_a\lambda(a,c) \notag
\end{align}
as required.

Checking the LHS is even simpler. We have

\begin{equation}
Hex^l = \frac{e^{b+c}}{ab} + \frac{e^c}{bc} + \frac{1}{ac} \notag
\end{equation}

Hence

\begin{align}
\bigl[ 1+\partial_a-\partial_b \bigr] Hex^l = & \frac{e^{b+c}}{ab} +
\frac{e^c}{bc} + \frac{1}{ac} \notag \\
& + ( -\frac{e^{b+c}}{a^2 b} - \frac{1}{a^2 c} ) \notag \\
& - (-\frac{e^{b+c}}{ab^2} + \frac{e^{b+c}}{ab} - \frac{e^c}{b^2 c})
\notag
\end{align}

After a small manipulation involving the $e^{b+c}$ terms and some
simplification we get the desired result.

\paragraph{Terms in $c[ab]$}

Here we have:

\begin{multline}
-\frac{e^{b+c}}{a^2 b} + \frac{e^c}{bc^2} + \frac{1}{ac} -
\frac{1}{bc^2} + \frac{1}{a^2b} - \frac{1}{ab} + \frac{1}{a^2b} -
\frac{1}{bc} - \frac{1}{bc^2} \notag \\
= \partial_a\lambda(a,b) e^{b+c} - \partial_c\lambda(c,b) e^c +
\partial_a\lambda(c,a) + \lambda(c,a) - \partial_c\lambda(a,c) \notag
\end{multline}

As with the $b[ab]$ terms, I claim that the $c[ab]$ equation is
simply the statement
\begin{equation}
\bigl[ 1+\partial_a-\partial_c \bigr] Hex^l = \bigl[ 1+\partial_a-
\partial_c \bigr] Hex^r \notag
\end{equation}
so that it, too, follows from the hexagon.  The calculation is
similar to the $b[ab]$ case and will not be repeated.

\paragraph{The 2FF Solution}

We have shown that the solution (\ref{Phi2FF}) takes the form

\begin{equation}
\Phi (a,b) = 1 + [ab] \lambda(a,b) - a[ab] \partial _a \lambda(a,b)
- b[ab] \partial _b \lambda(a,b) \label{NewPhi}
\end{equation}
where $\lambda$ is a solution of the 1FF equation.

Although this is technically only a solution to the positive
hexagon, it is not hard to show that in fact the unitarity condition
(\ref{unitary}) suffices to insure that the same solution also
satisfies the negative hexagon (of course this also follows on
general principles from arguments given in \cite{DBN2}, alluded to
earlier). Moreover, since in 2FF all commutators of commutators are
zero, Kurlin's argument (see \cite{Kurlin}, Proposition 5.10) still
applies to show that this unitarity condition also insures that the
solution satisfies the pentagon equation.

Thus each solution $\lambda$ to the 1FF equation gives rise to a
solution to the 2FF equation according to formula (\ref{NewPhi}).

\begin{remark}
Singular Solution
\end{remark}
As with the 1FF equation, one can verify that there is also a
singular solution $\lambda(x,y) = \frac{1}{xy}$.

\section{Concluding Remarks}

There are a number of directions in which the work of this paper
could be extended or applied.  An obvious possible extension is to
attempt to extend the results of the 1FF and 2FF quotients to `nFF'
quotients, with $n \in \N$.

More broadly, one could consider which other quotients might afford a closed-form solution.  For instance, there is a standard way to associate to any finite-type invariant of knots, such as the coefficients of the
Alexander polynomial, a `weight system', ie a linear functional from
the space of chord diagrams on the circle modulo the 4T relation to
$\Q$ (see \cite{DBN1}). The kernel of the Alexander polynomial, or of any other finite type invariant,
consists of those chord diagrams on which the weight system
corresponding to the Alexander polynomial or other invariant vanishes. One could therefore consider whether the
corresponding quotient also allows a closed--form formula for the
associator.

In a different direction, one can consider quotients which are
well-behaved under a suitable class of operations on knots, or
knot-like objects.  This is related to the `algebraic knot theory'
program initiated by Bar-Natan \cite{DBN6}, which considers
invariants of knot-like objects which transform `functorially' with
respect to certain operations between such objects. The operations
considered include the well-known strand-doubling and connected sum
operations, but also strand deletion and `unzip' operations.  A
broad class of knot properties can be characterized using these
operations, including knot genus, unknotting number and the property
of being ribbon. This being the case, invariants which are
well-behaved under these operations may shed further light on these
properties by allowing one to transform questions about knots into
questions which are set in the target space of the invariant, which
is presumably simpler and algebraically more tractable.

As noted earlier, research by Ng \cite{Ng} has shown in particular
that information about ribbon knots cannot be obtained only from a
finite approximation to a knot invariant, so that closed--form
solutions will be needed.  Here again, though, one is faced with the
fact that determination of closed-form formulas is extremely
difficult, and so consideration of suitable quotient space is
desirable.  This leads one to consider quotients which are
well-behaved under the allowed operations.  To give a hint of the
type of quotients this implies, we note that one can pictorially
represent the bracket of a Lie algebra as a trivalent graph:

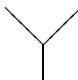
\[
\xy (-5,5)*{}; (0,0)*{} **\dir{-};  (5,5)*{}; (0,0)*{}
**\dir{-}; (0,-0)*{}; (0,-5)*{} **\dir{-}
\endxy
\]
where the three edges represent two inputs and one output.  Then,
for instance, the quotient explored by Kurlin, namely $\cL /\
[[\cL,\cL],[\cL,\cL]]$, is the quotient in which the following
diagram, and diagrams containing this diagram, are set to zero:

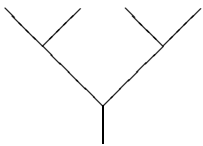
\[
\xy
(-13,5)*{}; (-8,0)*{} **\dir{-};
(-3,5)*{}; (-8,0)*{} **\dir{-};
(3,5)*{}; (8,0)*{} **\dir{-};
(13,5)*{}; (8,0)*{} **\dir{-};
(-8,0)*{}; (0,-8)*{} **\dir{-};
(8,0)*{}; (0,-8)*{} **\dir{-};
(0,-8)*{}; (0,-13)*{} **\dir{-}
\endxy
\]

It turns out that `internal' properties of a diagram are generally
well-behaved under the relevant operations.  For instance, one could
consider the quotient in which the following diagram, and all
diagrams containing this diagram, are set to zero:

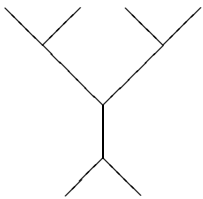
\[
\xy
(-13,5)*{}; (-8,0)*{} **\dir{-};
(-3,5)*{}; (-8,0)*{} **\dir{-};
(3,5)*{}; (8,0)*{} **\dir{-};
(13,5)*{}; (8,0)*{} **\dir{-};
(-8,0)*{}; (0,-8)*{} **\dir{-};
(8,0)*{}; (0,-8)*{} **\dir{-};
(0,-8)*{}; (0,-15)*{} **\dir{-};
(0,-15)*{}; (-5,-20)*{} **\dir{-};
(0,-15)*{}; (5,-20)*{} **\dir{-}
\endxy
\]

The key property of such diagrams for present purposes is that they
contain an internal vertex, ie a vertex whose edges lead only to
other (trivalent) vertices, and not to endpoints of the diagram. The
property of having such internal vertices is well-behaved under the
algebraic knot theory operations, and hence leads to a quotient
whose exploration would be worthwhile from the algebraic knot theory
viewpoint.


\begin{thebibliography}{9}

\bibitem[BN1]{DBN1}
Bar-Natan, D.: `On the Vassiliev Knot Invariants', \emph{Topology}
(Vol. 34:2), pp. 423-72 (1995).
%
\bibitem[BN2]{DBN2}
Bar-Natan, D.: `Non-Associative Tangles', \emph{AMS/IP Studies in
Advanced Mathematics} (Vol.2:I), pp. 139-83 (1997).
%
\bibitem[BN3]{DBN3}
Bar-Natan, D.: `Knot invariants, associators and a strange breed of
planar algebras,' talk at the Fields Institute, January 11th, 2001,
http://katlas.math.toronto.edu/~drorbn/Talks/Fields-010111/.
%
\bibitem[BN4]{DBN4}
Bar-Natan, D.: `Algebraic Structures on Knotted Objects and
Universal Finite Type Invariants,' web publication:
http://katlas.math.toronto.edu/~drorbn/papers/AlgebraicStructures
%
\bibitem[BN5]{DBN6}
Bar-Natan, D.: `Research Proposal on Knot Theory and Algebra,' web
publication:
http://katlas.math.toronto.edu/~drorbn/Profile/ResearchProposal-07.pdf
%
\bibitem[Drin]{Drinfel'd} Drinfel'd, V.G.: `On Quasi-Hopf Algebras',
\emph{Leningrad Math. J.} (Vol. 1), pp. 1419-57 (1990).
%
\bibitem[Drin2]{Drinfel'd2}
Drinfel'd, V.G.: `On Quasi-Triangular Quasi-Hopf Algebras and a
Group Closely Connected With Gal($\bar{\Q}/\Q$)', \emph{Leningrad
Math. J.} (Vol. 2), pp. 829-60 (1991).
%
\bibitem[Kur]{Kurlin}
Kurlin, V.: `Explicit Description of Compressed Logarithms of all
Drinfel'd Associators', \emph{Journal of Algebra} (Vol. 292, no. 1),
pp. 184-242 (2005); arXiv:math.GT/0408398.
%
\bibitem[Lieb]{Lieberum}
Lieberum, J.: `The Drinfel'd Associator of gl(1$|$1)',
arXiv:math.QA/0204346.
%
\bibitem[Ng]{Ng}
Ng, K.Y.: `Groups of Ribbon Knots,' \emph{Topology} 37 (1998)
441-58, arXiv:q-alg/9502017 (addendum at arXiv:math.GT/0310074).
%
\bibitem[Thu]{Thurston}
Thurston, Dylan P.: `The algebra of knotted trivalent graphs and
Turaev's shadow world,' Geom. Topol. Monogr. 4 (2002) 337-362;
math.GT/0311458.
%
\bibitem[Zin]{Zin}
Zinbiel, G.W.: `Encyclopedia of Types of Algebras 2010', arXiv:1101.0267 (math.RA).

\end{thebibliography}
\end{document}